\documentclass[11pt]{article}
\usepackage[a4paper,margin=2.4cm]{geometry}
\usepackage{amsmath,amssymb,amsthm,mathtools}
\usepackage{booktabs,array,enumitem,microtype}
\usepackage[mathlines]{lineno}
\usepackage{tikz}
\usetikzlibrary{positioning,calc}
\usepackage[hidelinks]{hyperref}
\usepackage{cleveref}
\newtheorem{theorem}{Theorem}[section]
\newtheorem{proposition}[theorem]{Proposition}
\newtheorem{lemma}[theorem]{Lemma}
\newtheorem{corollary}[theorem]{Corollary}
\newtheorem{conjecture}[theorem]{Conjecture}
\theoremstyle{definition}\newtheorem{definition}[theorem]{Definition}

\theoremstyle{remark}\newtheorem{remark}[theorem]{Remark}
\newcommand{\psq}{\psi_{sq}}\newcommand{\pq}{\psi_q}\newcommand{\bii}{\beta_2}
\title{Sub-quorum colorings of graphs}
\author{\begin{tabular}{c}
Haobo Ma$^{1,2}$ \qquad Rafik Sahbi$^3$ \qquad Wenlin Zhang$^{1,4}$\\[6pt]
\small $^1$The Omega Institute\\
\small $^2$ChronoAI Pte Ltd\\
\small $^3$Department of Fundamental Science and Technology\\
\small National Higher School of Advanced Technologies\\
\small B.P. 474, Martyrs Square, Algiers 16001, Algeria\\
\small $^4$National University of Singapore\\
\small 21 Lower Kent Ridge Road, Singapore 119077\\[4pt]
\small \texttt{auric@aelf.io}; \texttt{r.sahbi@g.essa-alger.edu.dz}\\
\small \texttt{e1327962@u.nus.edu}
\end{tabular}}
\date{}
\begin{document}
\maketitle
\begin{abstract}
A sub-quorum coloring is a partial vertex coloring in which every colored vertex sees at least half of its colored closed neighborhood in its own color. We study the maximum number of colors $\psq(G)$ from structural, algorithmic and extremal viewpoints. Every $2$-independent set yields a sub-quorum coloring, so $\psq(G)\ge\beta_2(G)$. We prove the degree bound $\psq(G)\le \Delta(G)|V(G)|/(\Delta(G)+\delta(G)-1)$ for $\delta(G)\ge2$, and obtain an exact matching characterization for graphs of maximum degree at most three. Consequently every cubic graph of order $n$ satisfies $n/2\le\psq(G)\le3n/5$, with a stability identity at the upper endpoint. For trees we give an edge representation and a linear-time dynamic program that computes both $\psq(T)$ and an optimal coloring; we also show on the double star $S_{2,2}$ that $\psq$ can exceed the dissociation number even on trees. We prove NP-completeness even for bipartite graphs and determine exact values for several classical families. For rectangular grids, we formulate the conjecture $\psq(G_{m,n})=\beta_2(G_{m,n})=F(m,n)$, recover the dissociation-number formula from known $3$-path vertex-cover results, prove the ladder case directly, and verify the conjecture for every fixed width $2\le m\le11$; widths $3$ through $11$ use exact transfer certificates. For hypercubes we prove $\psq(Q_n)=\beta_2(Q_n)=2^{n-1}$ for every $n\ge2$. Finally, a Nordhaus--Gaddum argument gives $\psq(G)+\psq(\overline G)>\log_2 |V(G)|$ and leads to a complementary upper-bound problem.
\end{abstract}

\section{Introduction}
Quorum colorings were introduced by Hedetniemi, Hedetniemi, Laskar and Mulder \cite{HHLM2013}. A partition of $V(G)$ is a quorum coloring when each vertex has a weak majority of its closed neighborhood in its own color class. The same paper proposed the sub-quorum variant, in which vertices may be left uncolored; equivalently, one quorum-colors an induced subgraph. This additional freedom makes the colored support itself part of the optimization problem.

The original quorum condition is closely related to defensive alliances \cite{Fricke2003,HaynesLachniet2007,ErohGera2012}. Subsequent work settled several of the open problems posed for quorum colorings \cite{SahbiChellali2018,Sahbi2020,Sahbi2021,SahbiErratum2021}, and exact or algorithmic results are known for several classes of trees \cite{Sahbi2022,SahbiBoumalhaIssad2025,SahbiIssad2026}. Sub-quorum colorings have received much less attention. A preprint by Sahbi, Belkina and Bennadji treated several infinite families of caterpillars \cite{SahbiBelkinaBennadji2021}. The purpose of the present paper is to develop the sub-quorum parameter itself through structural bounds, algorithms, complexity, exact evaluations and complementary-graph inequalities.

A central comparison is with the dissociation number $\beta_2(G)$, the maximum size of a vertex set inducing maximum degree at most one. Every dissociation set can be colored with private colors, giving $\psq(G)\ge\beta_2(G)$. We first combine this observation with degree counting. For graphs of maximum degree at most three, the local majority requirement reduces to the existence of a same-colored neighbor at precisely those colored vertices having at least two colored neighbors; this yields an exact matching characterization and sharp bounds for cubic graphs.

Trees admit a different simplification. Since every induced subgraph is a forest, monochromatic edges encode the color classes without using color names. This leads to a four-state rooted-tree dynamic program and an $O(n)$ algorithm for both the optimum value and an optimal coloring. The double star $S_{2,2}$ shows that equality with $\beta_2$ is not a general tree phenomenon.

We also study computational complexity and several standard graph families. The decision problem is NP-complete even on bipartite graphs. This makes exact results on structured bipartite families particularly useful: for rectangular grids we determine $\psq(G_{m,n})$ for every fixed width $2\le m\le11$, with a direct proof for ladders and exact transfer certificates for widths $3$ through $11$. For hypercubes, a spectral argument based on Huang's signed adjacency matrix \cite{Huang2019} gives the exact identity $\psq(Q_n)=\beta_2(Q_n)=2^{n-1}$ for every $n\ge2$. Finally, complementing the structural results, we initiate a Nordhaus--Gaddum analysis of $\psq$ and obtain a logarithmic universal lower bound for $\psq(G)+\psq(\overline G)$.

The paper is organized as follows. After definitions, we establish the fundamental dissociation lower bound, degree bounds and the subcubic matching characterization. We then treat trees, complexity and exact values for classical families, followed by rectangular grid strips and hypercubes. The final mathematical section gives the Nordhaus--Gaddum lower bound and records the corresponding upper-bound conjecture.

\section{Definitions and notation}
All graphs are finite, simple and undirected. For a graph $G=(V,E)$ and $v\in V$, let $N_G(v)$ and $N_G[v]=N_G(v)\cup\{v\}$ denote the open and closed neighborhoods, and let $d_G(v)=|N_G(v)|$. For $S\subseteq V$, $G[S]$ is the subgraph induced by $S$ and $d_S(v)=|N_G(v)\cap S|$. We write $\Delta(G)$ and $\delta(G)$ for the maximum and minimum degrees. The complement is $\overline G$.

A matching is a set of pairwise vertex-disjoint edges; its maximum size is $\mu(G)$. For $k\ge1$, a set $S\subseteq V$ is \emph{$k$-independent} if $\Delta(G[S])\le k-1$. Its maximum cardinality is $\beta_k(G)$. Thus $\beta_1(G)=\alpha(G)$ is the ordinary independence number, while a $2$-independent set induces only isolated vertices and edges. We write $\omega(G)$ for the clique number.

\begin{definition}
A \emph{quorum coloring} of $G$ is a partition $\pi=\{V_1,\ldots,V_k\}$ of $V(G)$ such that, for every $v\in V_i$,
\[
 |N_G[v]\cap V_i|\ge \frac{|N_G[v]|}{2}.
\]
The maximum possible $k$ is the \emph{quorum coloring number} $\pq(G)$.
\end{definition}
Each quorum class is a defensive alliance; equivalently, quorum colorings are alliance partitions \cite{HHLM2013,HaynesLachniet2007}.

\begin{definition}
A \emph{sub-quorum coloring} is an onto partial map $f:V(G)\dashrightarrow\{1,\ldots,k\}$. If $S=\operatorname{dom}(f)$, then every $v\in S$ must satisfy
\[
 |\{u\in N_G[v]\cap S:f(u)=f(v)\}|\ge\frac{|N_G[v]\cap S|}{2}.
\]
The maximum possible number of colors is denoted $\psq(G)$. We write $G_f=G[S]$.
\end{definition}
Thus a sub-quorum coloring of $G$ is exactly a quorum coloring of some induced subgraph $G[S]$, and consequently
\begin{equation}\label{eq:induced}
\psq(G)=\max_{S\subseteq V(G)}\pq(G[S])\ge\pq(G).
\end{equation}

\begin{lemma}[local form]\label{lem:local}
Let $f$ be a sub-quorum coloring with colored set $S$. For $v\in S$, let $a(v)$ be the number of colored neighbors with color $f(v)$ and $b(v)$ the number of colored neighbors with a different color. Then $b(v)\le a(v)+1$. In particular, if $v$ is a singleton color class, then $d_S(v)\le1$.
\end{lemma}
\begin{proof} The quorum inequality is $2(a(v)+1)\ge a(v)+b(v)+1$, which is equivalent to the assertion. \end{proof}

We shall use the standard fact from \cite{HHLM2013} that $\pq(G)=|V(G)|$ iff $\Delta(G)\le1$. We also record why optimal color classes are connected. Splitting a color class into its induced connected components changes no same-colored adjacency and leaves the colored support fixed, so every local inequality is preserved. A disconnected class would therefore increase the number of colors, contradicting optimality. This applies to both quorum and sub-quorum colorings.

For upper bounds, let $\Omega(G)$ be the maximum of $|T|+|M|$ over matchings $M$ and vertex sets $T$ disjoint from $V(M)$ such that every $t\in T$ has at most one neighbor in $T\cup V(M)$. We call $(T,M)$ a \emph{feasible representative pair}.
\begin{lemma}\label{lem:omega} For every graph $G$, $\psq(G)\le\Omega(G)$. \end{lemma}
\begin{proof}
Take an optimal sub-quorum coloring. Let $T$ consist of its singleton color classes and choose one edge from every other class, forming a matching $M$; such an edge exists because the class is connected. Each $t\in T$ has at most one neighbor in the colored support by Lemma~\ref{lem:local}, and hence at most one in $T\cup V(M)$. Thus $(T,M)$ is feasible and its objective is exactly the number of colors.
\end{proof}

\section{A fundamental lower bound}
\begin{theorem}\label{thm:beta2} For every graph $G$, $\psq(G)\ge\bii(G)$. \end{theorem}
\begin{proof}
Let $S$ be a maximum $2$-independent set. Give every vertex of $S$ its own color and leave $V(G)\setminus S$ uncolored. Since $\Delta(G[S])\le1$, every colored closed neighborhood has size at most two, and the vertex itself supplies at least half of it. Hence the coloring uses $|S|=\bii(G)$ colors.
\end{proof}
This improves the immediate bound $\psq(G)\ge\beta_1(G)$ noted in \cite{HHLM2013}.

\subsection{Degree bounds and subcubic graphs}
The lower bound above combines naturally with degree information. We first give an upper bound that is valid without any restriction on the maximum degree beyond finiteness.

\begin{theorem}\label{thm:degree-upper}
Let $G$ be a graph of order $n$ with $\delta(G)\ge2$. Then
\[
\boxed{\psq(G)\le
\frac{\Delta(G)}{\Delta(G)+\delta(G)-1}\,n.}
\]
In particular, if $G$ is $r$-regular with $r\ge2$, then
\[
\psq(G)\le \frac{r}{2r-1}n.
\]
\end{theorem}
\begin{proof}
Consider a sub-quorum coloring using $s$ singleton classes and $t$ non-singleton classes, and let $u$ vertices be uncolored. By Lemma~\ref{lem:local}, a singleton-colored vertex has at most one colored neighbour, hence at least $\delta(G)-1$ uncolored neighbours. Counting edges between singleton-colored and uncolored vertices gives
\[
(\delta(G)-1)s\le \Delta(G)u.
\]
Every non-singleton class contains at least two vertices, so
\[
n\ge u+s+2t.
\]
Writing $q=s+t$ and using $u\ge(\delta(G)-1)s/\Delta(G)$, we obtain
\[
n\ge 2q+u-s
\ge 2q-\frac{\Delta(G)-\delta(G)+1}{\Delta(G)}s.
\]
Since $s\le q$, this implies
\[
n\ge \frac{\Delta(G)+\delta(G)-1}{\Delta(G)}q,
\]
which is the claimed inequality. The regular case follows by setting $\delta(G)=\Delta(G)=r$.
\end{proof}

For a colored set $S\subseteq V(G)$, define
\[
B(S)=\{v\in S:d_S(v)\ge2\}.
\]
The next result gives an exact description when $G$ is subcubic. It is specific to maximum degree three: if $d_S(v)\ge4$, the local quorum inequality may require two or more same-colored neighbours, so the matching reduction below no longer applies.

\begin{theorem}[Subcubic matching characterization]\label{thm:subcubic-characterization}
If $\Delta(G)\le3$, then
\[
\boxed{\psq(G)=\max_{S\subseteq V(G)}
\bigl(|S|-|B(S)|+\mu(G[B(S)])\bigr).}
\]
\end{theorem}
\begin{proof}
Fix $S\subseteq V(G)$. For $v\in S$, let $C(v)$ denote its color class. The sub-quorum condition is equivalent to
\[
1+d_{C(v)}(v)\ge \frac{1+d_S(v)}2,
\qquad\text{that is,}\qquad
 d_{C(v)}(v)\ge\left\lfloor\frac{d_S(v)}2\right\rfloor.
\]
Since $d_S(v)\le3$, the vertices of $B(S)$ are exactly those that require at least one same-colored neighbour; vertices in $S\setminus B(S)$ may remain singleton classes.

Let a sub-quorum coloring with colored set $S$ have color classes $C_1,\ldots,C_q$, and put $b_i=|C_i\cap B(S)|$ and $c_i=|C_i|$. We claim
\[
c_i-1\ge b_i-\mu(G[B(S)\cap C_i]).
\]
If $b_i<c_i$, this is immediate from $c_i-1\ge b_i$. If $b_i=c_i>0$, every vertex of $C_i$ requires a same-colored neighbour, so $G[C_i]=G[B(S)\cap C_i]$ has an edge and therefore its matching number is at least one; the displayed inequality follows again. (When $b_i=0$ it is trivial.) Summing over the classes, and observing that the union of maximum matchings in the pairwise disjoint sets $B(S)\cap C_i$ is a matching of $G[B(S)]$, gives
\[
|S|-q=\sum_i(c_i-1)\ge |B(S)|-\mu(G[B(S)]).
\]
Thus
\[
q\le |S|-|B(S)|+\mu(G[B(S)]).
\]

Conversely, let $M$ be a maximum matching of $G[B(S)]$. Give the two ends of every edge of $M$ a common color. The unmatched vertices of $B(S)$ form an independent set, since otherwise $M$ could be enlarged. Every unmatched vertex has a neighbour in $S$: if it has a neighbour in $S\setminus B(S)$, merge it with one such neighbour; otherwise all its neighbours in $S$ lie in $B(S)$ and are covered by $M$, so merge the vertex into the color class of one adjacent matching edge. A vertex of $S\setminus B(S)$ has at most one neighbour in $S$, so it cannot be used by two unmatched vertices. Vertices of $S\setminus B(S)$ not used in this procedure receive private colors. Every vertex of $B(S)$ then has a same-colored neighbour, while vertices outside $B(S)$ require none. Starting with $|S\setminus B(S)|$ private colors and $|M|$ matching colors, each unmatched vertex is added to an existing class or forms a pair with one distinct low-degree vertex; the number of colors is unchanged. The resulting coloring has exactly
\[
|S|-|B(S)|+|M|
\]
colors. Maximizing over $S$ proves the formula.
\end{proof}

The dissociation number satisfies the general bound
\[
\bii(G)\ge \frac{|V(G)|}{\left\lceil(\Delta(G)+1)/2\right\rceil},
\]
proved in \cite{Bresar2013}. Together with Theorems~\ref{thm:beta2} and~\ref{thm:degree-upper}, this immediately gives the following sharp interval for cubic graphs.

\begin{corollary}[Cubic graphs]\label{cor:cubic-bounds}
If $G$ is a cubic graph of order $n$, then
\[
\boxed{\frac n2\le\bii(G)\le\psq(G)\le\frac{3n}{5}.}
\]
The constants $1/2$ and $3/5$ cannot be improved: $K_4$ attains the lower endpoint, while the Petersen graph attains the upper endpoint.
\end{corollary}
\begin{proof}
The first inequality follows from the cited dissociation bound with $\Delta(G)=3$, the middle inequality is Theorem~\ref{thm:beta2}, and the last inequality is Theorem~\ref{thm:degree-upper}. For $K_4$, Proposition~\ref{prop:complete} gives $\psq(K_4)=2=n/2$. For the Petersen graph $P$, label the outer cycle $u_0u_1u_2u_3u_4u_0$, the inner vertices $v_0,\ldots,v_4$, with spokes $u_iv_i$ and inner edges $v_iv_{i+2}$ (indices modulo $5$). Then $\{u_0,u_1,u_3,v_2,v_3,v_4\}$ induces three disjoint edges, and is therefore a $2$-independent set of size $6$. Hence $\psq(P)\ge6$, while the upper bound gives $\psq(P)\le6$.
\end{proof}

The proof of the cubic upper bound also records how equality, or near equality, can occur.
\begin{proposition}[Stability identity]\label{prop:cubic-stability}
Let $G$ be cubic of order $n$, and let a sub-quorum coloring use $q=s+t$ colors, where $s$ classes are singletons and $t$ are non-singletons. Let $u$ be the number of uncolored vertices, set
\[
e=n-(u+s+2t)\ge0,
\]
and let
\[
\sigma=3u-2s\ge0.
\]
Then
\[
\boxed{3n-5q=\sigma+t+3e.}
\]
Consequently, if $\psq(G)=3n/5$, then every class in an optimal coloring is a singleton and
\[
\bii(G)=\psq(G)=\frac{3n}{5}.
\]
\end{proposition}
\begin{proof}
The inequality $2s\le3u$ is the edge count used in Theorem~\ref{thm:degree-upper}, so $\sigma\ge0$. By definition of $e$,
\[
3n-5q=3(u+s+2t+e)-5(s+t)
=(3u-2s)+t+3e=\sigma+t+3e.
\]
If $q=3n/5$, all three nonnegative terms on the right vanish. In particular $t=0$, so the colored set induces maximum degree at most one by Lemma~\ref{lem:local}. It is therefore a $2$-independent set of size $q$, and Theorem~\ref{thm:beta2} gives the asserted equality.
\end{proof}

\section{Trees: an edge representation and a linear-time algorithm}\label{sec:trees}
For trees, the absence of cycles permits a particularly economical description of a sub-quorum coloring.  If $X$ is the colored set, let $F$ be the set of edges of $T[X]$ whose ends have the same color.

\begin{lemma}[Edge representation]\label{lem:tree-edge}
For every tree $T$,
\[
\psq(T)=\max\left\{|X|-|F|:\begin{array}{l}
X\subseteq V(T),\ F\subseteq E(T[X]),\\[1mm]
d_{T[X]\setminus F}(v)\le d_F(v)+1\quad(v\in X)
\end{array}\right\}.
\]
Moreover, every feasible pair $(X,F)$ yields a sub-quorum coloring with $|X|-|F|$ colors by assigning one color to each component of the forest $(X,F)$.
\end{lemma}
\begin{proof}
Start with an optimal sub-quorum coloring and split every disconnected color class into its connected components.  Feasibility is unchanged and the number of colors cannot decrease, so we may assume that every color class is connected.  Since $T[X]$ is a forest, the monochromatic edges form a forest $(X,F)$ whose components are precisely the color classes; hence their number is $|X|-|F|$.  At a colored vertex $v$, the $F$-edges are exactly the same-colored neighbors and the remaining edges of $T[X]$ go to differently colored vertices.  Lemma~\ref{lem:local} therefore gives the displayed inequality.

Conversely, give the components of a feasible forest $(X,F)$ distinct colors.  An edge of $T[X]\setminus F$ joins distinct components, so the displayed inequality is exactly the local sub-quorum condition.  The number of colors is $|X|-|F|$.
\end{proof}

We next obtain a linear-time dynamic program.  Root $T$ at $r$.  For a vertex $v$, let $T_v$ be its rooted subtree.  Four states are used: $U_v$ when $v$ is uncolored; $C_v$ when $v$ is colored and its parent is absent from the colored set; $S_v$ when the parent edge is monochromatic; and $D_v$ when the parent edge is colored but bichromatic.  Write $X_v=X\cap V(T_v)$ and $F_v=F\cap E(T_v)$.  The states $U_v,C_v,D_v$ maximize $|X_v|-|F_v|$, while $S_v$ maximizes $|X_v|-|F_v|-1$, charging its monochromatic parent edge to the child.  Feasibility includes the prescribed parent state, with all vertices of $T_v$ satisfying the local condition.  Every edge of $F$ is therefore charged exactly once.  The root has no parent edge and only states $U_r,C_r$ are used there.  Thus for a leaf
\[
(U_v,C_v,S_v,D_v)=(0,1,0,1).
\]
If $v$ is uncolored, then
\begin{equation}\label{eq:tree-U}
U_v=\sum_{u\in\operatorname{ch}(v)}\max\{U_u,C_u\}.
\end{equation}
When $v$ is colored, each child is put in state $U,S$, or $D$.  If $s$ and $d$ denote the numbers of children put in states $S$ and $D$, respectively, define
\[
M_v(c)=\max\left\{\sum_{u\in\operatorname{ch}(v)}X_u:X_u\in\{U_u,S_u,D_u\},\ d-s\le c\right\}.
\]
The local condition gives
\begin{equation}\label{eq:tree-states}
C_v=1+M_v(1),\qquad S_v=M_v(2),\qquad D_v=1+M_v(0).
\end{equation}

\begin{lemma}[Unit-state lemma]\label{lem:tree-unit}
For every vertex $v$,
\[
C_v-1\le S_v\le D_v\le C_v,
\qquad U_v\ge C_v-1.
\]
Consequently $S_v-C_v,D_v-C_v\in\{-1,0\}$.
\end{lemma}
\begin{proof}
We argue by induction from the leaves.  Assume the statement for all children of $v$.  Monotonicity gives $M_v(c)\le M_v(c+1)$.  If an optimizer for $M_v(c+1)$ is not feasible for $M_v(c)$, then its integer balance satisfies $d-s=c+1$.  For $c=0,1$ this forces at least one child to be chosen in state $D$.  Replacing one such $D$ choice by the corresponding $S$ choice decreases $d-s$ by two and loses at most one unit, because $0\le D_u-S_u\le1$ by induction.  The modified choice is feasible for $M_v(c)$, and therefore
\[
0\le M_v(c+1)-M_v(c)\le1\quad(c=0,1).
\]
Applying the same replacement directly to an optimizer for $M_v(2)$ whenever it is not feasible for $M_v(0)$ gives $M_v(2)-M_v(0)\le1$.
Together with \eqref{eq:tree-states}, these inequalities give
$C_v-1\le S_v\le D_v\le C_v$.  Finally every contribution allowed in $M_v(1)$ is at most $\max\{U_u,C_u\}$ by the induction hypothesis, so $M_v(1)\le U_v$ by \eqref{eq:tree-U}.  Thus $U_v\ge C_v-1$.
\end{proof}

The preceding lemma removes the apparent knapsack.  Children with $U_u\ge C_u$ are optimally left in state $U$ whenever a colored choice is not needed for balance.  If $U_u<C_u$, then $U_u=C_u-1$, and only the following three normalized types can occur:
\[
(U_u,S_u,D_u)-(C_u,C_u,C_u)
\in\{(-1,-1,-1),(-1,-1,0),(-1,0,0)\}.
\]
Let $A_v$ count children of type $(-1,-1,0)$, let $B_v$ count those of type $(-1,0,0)$, and let $R_v$ count children satisfying $S_u=U_u=D_u$. These are exactly type $(-1,-1,-1)$ and the case $U_u=S_u=D_u=C_u$; in particular the $A_v$ and $R_v$ children are disjoint. Put $W_v=\sum U_u$.

We justify the maximization for $c\in\{0,1,2\}$. For each $B_v$ child, state $S$ dominates both $U$ and $D$ in value and balance; for each $R_v$ child, state $S$ has the same value and improves balance. For each $A_v$ child, state $U$ can be replaced by $S$ at no loss. Every remaining child has $U_u\ge C_u$, $S_u<U_u$ and $D_u\le U_u$, so state $D$ can be replaced by $U$ at no loss and with improved balance. Thus such a child is either in state $U$ or supplies one unit of balance through an $S$ choice costing at least one unit of value.

Any remaining paid $S$ choice can be changed to $U$, gaining at least one unit. If that change violates $d-s\le c$, at least one $A_v$ child must currently be in state $D$, since all other chosen colored states contribute nonpositive balance and $c\ge0$. Change one such $D$ choice to $S$, decreasing balance by two at a cost of one unit. The combined change is feasible and does not decrease the value. Repeating eliminates all paid $S$ choices. Consequently an optimum can use all $B_v$ and $R_v$ children in state $S$, with each $A_v$ child in state $S$ or $D$, and all other children in state $U$. If $z$ of the $A_v$ children are chosen in state $D$, its value is $W_v+B_v+z$ and its balance is $2z-A_v-B_v-R_v$, giving
\[
2z\le A_v+B_v+R_v+c.
\]
Maximizing $z$ subject to this inequality and $0\le z\le A_v$ proves
\begin{equation}\label{eq:tree-closed}
M_v(c)=W_v+B_v+
\min\left\{A_v,\left\lfloor\frac{A_v+B_v+R_v+c}{2}\right\rfloor\right\},
\qquad c\in\{0,1,2\}.
\end{equation}
Thus all four states at $v$ are obtained by one scan of its children.

\begin{theorem}[Trees]\label{thm:tree-linear}
For every tree $T$ of order $n$, the value $\psq(T)$ and an optimal sub-quorum coloring can be computed in $O(n)$ time and $O(n)$ space.
\end{theorem}
\begin{proof}
Process the vertices in postorder.  Equation~\eqref{eq:tree-U}, the counters $W_v,A_v,B_v,R_v$, and then \eqref{eq:tree-closed} and \eqref{eq:tree-states} require $O(|\operatorname{ch}(v)|+1)$ time.  Summing over the tree gives $O(n)$ time and four state values plus constant auxiliary data per vertex give $O(n)$ space.  At the root,
\[
\psq(T)=\max\{U_r,C_r\}.
\]
Fix an order of the children at every vertex.  For $U_v$, store the maximizing choice $U_u$ or $C_u$ for each child, breaking ties deterministically.  For each $c\in\{0,1,2\}$, store
\[
z_v(c)=\min\left\{A_v,\left\lfloor\frac{A_v+B_v+R_v+c}{2}\right\rfloor\right\}.
\]
To reconstruct an $M_v(c)$ choice, put every $B_v$- and $R_v$-child in state $S$, the first $z_v(c)$ $A_v$-children in state $D$, the other $A_v$-children in state $S$, and all remaining children in state $U$.  Use $c=1,2,0$ for states $C_v,S_v,D_v$, respectively.  Trace these choices once from a maximizing root state and put a parent edge in $F$ exactly when its child is traced in state $S$.  The stored choices and the reconstruction use $O(n)$ space and time; the components of $(X,F)$ then give the color classes by Lemma~\ref{lem:tree-edge}.
\end{proof}

The lower bound $\psq\ge\bii$ can already be strict on trees.
\begin{proposition}\label{prop:s22}
For the double star $S_{2,2}$,
\[
\bii(S_{2,2})=4,\qquad \psq(S_{2,2})=5.
\]
\end{proposition}
\begin{proof}
Let $x,y$ be the adjacent centers, each with two leaves.  The four leaves form a dissociation set, while any dissociation set containing a center has size at most four; hence $\bii(S_{2,2})=4$.  For the sub-quorum coloring, give $x$ and $y$ one common color and give the four leaves four distinct colors.  At either center there is one same-colored and two differently colored neighbors, so the local inequality holds and five colors are feasible.  Six colors would make every vertex a singleton, which fails at each center.  Thus $\psq(S_{2,2})=5$.
\end{proof}

\section{Complexity}
Consider \textsc{Sub-Quorum-$K$}: given $G$ and $K\le |V(G)|$, decide whether $\psq(G)\ge K$. A certificate consists of the colored set and its color labels, and all local inequalities can be checked in polynomial time; hence the problem lies in NP. We prove the following strengthening.
\begin{theorem}\label{thm:np} \textnormal{\textsc{Sub-Quorum-$K$}} is NP-complete, even for bipartite graphs. \end{theorem}
\begin{proof}
Membership in NP follows from the preceding certificate argument. For NP-hardness we reduce from \textsc{Exact Cover by 3-Sets} (X3C), which is NP-complete \cite{GareyJohnson1979}. Let $(X,\mathcal C)$ be an instance, where $|X|=3q$ and $\mathcal C=\{C_1,\ldots,C_m\}$ with $|C_i|=3$. We may assume $m\ge q$. Put
\[
R=2m+1,\qquad L=3R-1.
\]
For each $C_i$ introduce a selector $s_i$. For $1\le j\le L$ introduce $p_{i,j},a_{i,j},b_{i,j}$ and the edges
$s_ip_{i,j},p_{i,j}a_{i,j},p_{i,j}b_{i,j}$. For every $x\in X$ introduce $R$ independent copies $x^1,\ldots,x^R$, and join $s_i$ to every $x^r$ whenever $x\in C_i$. No other edges are added. The resulting graph $H$ is bipartite. Set
\[
B_0=2Lm+3qR,\qquad K=B_0+q.
\]
If $\mathcal C$ has an exact cover $\mathcal C'$, give each $a_{i,j},b_{i,j}$ and each element-copy $x^r$ a private color. For $C_i\in\mathcal C'$, color $s_i,p_{i,1},\ldots,p_{i,L}$ with one new common color, leaving those vertices uncolored for $C_i\notin\mathcal C'$. Every element-copy then sees exactly one colored selector. A selected selector has $L+1=3R$ same-colored vertices in its closed colored neighborhood and $3R$ differently colored neighbors, while each selected $p_{i,j}$ has exactly two same-colored vertices among four colored vertices in its closed neighborhood. Hence this is a sub-quorum coloring with $K$ colors.

Conversely suppose $H$ has a sub-quorum coloring with at least $K$ colors, and replace it by an optimal one. Let $T$ be the number of colored selectors and, for $x\in X$, let $t_x$ be the number of colored selectors $s_i$ with $x\in C_i$. For the private gadget
\[
P_i=\{s_i\}\cup\{p_{i,j},a_{i,j},b_{i,j}:1\le j\le L\},
\]
at most $2L$ colors occur if $s_i$ is uncolored and at most $2L+1$ if it is colored. Indeed, if $s_i$ is uncolored, each triple $p_{i,j},a_{i,j},b_{i,j}$ contributes at most two colors: three distinct colors would violate the local inequality at $p_{i,j}$. If $s_i$ is colored, count its color once. A triple contributes at most two further colors, since three distinct colors different from that of $s_i$ again leave $p_{i,j}$ with no same-colored neighbor and at least two differently colored neighbors. Summing over triples and then gadgets proves the bound $2Lm+T$; shared colors can only reduce this sum. If $t_x\ge2$, no element-copy of $x$ can be a singleton color class: it would have at least two colored selector-neighbors, contradicting Lemma~\ref{lem:local}. Since the element-copies are pairwise nonadjacent and every optimal color class is connected, such an $x$ contributes no new color supported only on its copies. Any color class containing an element-copy and at least one other vertex must contain a selector adjacent to that copy, so it has already been counted on a private gadget. Therefore the only remaining colors are singleton classes supported on element-copies. For a fixed $x$, there are at most $R$ such colors when $t_x\le1$, and none when $t_x\ge2$. If
\[
c=|\{x\in X:t_x\ge2\}|,
\]
then
\[
\psq(H)\le 2Lm+T+(3q-c)R=B_0+T-cR.
\]
If $c\ge1$, using $T\le m$ and $R=2m+1$ gives $\psq(H)\le B_0+m-R<B_0+q=K$, a contradiction. Hence $c=0$: the $T$ sets represented by colored selectors are pairwise disjoint, so $T\le q$. The same inequality now gives
\[
K\le\psq(H)\le B_0+T\le B_0+q=K,
\]
and therefore $T=q$. The corresponding $q$ disjoint 3-sets contain all $3q$ elements and form an exact cover. The reduction is polynomial and preserves bipartiteness.
\end{proof}

\section{Exact values and bounds for classical families}
The following results are useful benchmarks and also show repeatedly that $\bii$ is the correct parameter.
\begin{theorem}\label{thm:pathcycle}
For $n\ge1$, $\psq(P_n)=\bii(P_n)=\lceil2n/3\rceil$. For $n\ge3$, $\psq(C_n)=\bii(C_n)=\lfloor2n/3\rfloor$.
\end{theorem}
\begin{proof}
Write $P_n=v_1v_2\cdots v_n$. In any sub-quorum coloring, three consecutive vertices cannot all be colored with three distinct colors: if $v_i,v_{i+1},v_{i+2}$ had distinct colors, then $v_{i+1}$ would see only one vertex of its own color in a colored closed neighborhood of size three. Partitioning the path into consecutive blocks of three, with a final block of size one or two, therefore gives
\[
 \psq(P_n)\le \left\lceil\frac{2n}{3}\right\rceil.
\]
Conversely, color exactly the vertices $v_i$ with $i\not\equiv0\pmod3$, assigning a different color to every selected vertex. Each selected vertex has at most one selected neighbor, so the selected set is $2$-independent and Lemma~\ref{lem:local} shows that this is a sub-quorum coloring. Its size is $\lceil2n/3\rceil$, proving the path formula and the equality with $\bii(P_n)$.

Now let $n\ge3$. On the cycle $v_1\cdots v_nv_1$, select the vertices $v_i$ with $1\le i\le n-1$ and $i\not\equiv0\pmod3$, leaving $v_n$ uncolored. This is a $2$-independent set of size $\lceil2(n-1)/3\rceil=\lfloor2n/3\rfloor$, giving the same lower bound for $\psq(C_n)$. For the converse, take an optimal sub-quorum coloring. If some vertex is uncolored, deleting it produces a path $P_{n-1}$ carrying the same colors, and hence
\[
 \psq(C_n)\le \psq(P_{n-1})=\left\lceil\frac{2(n-1)}3\right\rceil=\left\lfloor\frac{2n}3\right\rfloor.
\]
If all vertices are colored, every vertex has a colored closed neighborhood of size three. Hence the quorum condition forces every vertex to have at least one neighbour of its own color. Every monochromatic component on the cycle therefore contains at least two vertices, so the number of colors is at most $\lfloor n/2\rfloor\le\lfloor2n/3\rfloor$. This proves the required upper bound without any decoloring argument. Thus $\psq(C_n)=\lfloor2n/3\rfloor$. The same periodic construction and the standard three-consecutive-vertices bound give $\bii(C_n)=\lfloor2n/3\rfloor$.
\end{proof}

We next record exact values for complete and complete multipartite graphs.
\begin{proposition}\label{prop:complete}
For every $n\ge2$, $\psq(K_n)=2$.
\end{proposition}
\begin{proof}
Every nonempty induced subgraph of $K_n$ is a complete graph. Hedetniemi et al.~\cite{HHLM2013} proved that the quorum coloring number of a complete graph is at most two, while an induced $K_2$ has quorum coloring number two. Equation~\eqref{eq:induced} therefore gives $\psq(K_n)=2$.
\end{proof}

\begin{proposition}\label{prop:star}
For every $n\ge2$, $\psq(K_{1,n})=n$.
\end{proposition}
\begin{proof}
The $n$ leaves form an independent set, so $\psq(K_{1,n})\ge n$. A sub-quorum coloring with $n+1$ colors would make every vertex a singleton-colored vertex; the center would then have $n$ colored neighbors of other colors and would violate Lemma~\ref{lem:local}. Hence $\psq(K_{1,n})\le n$.
\end{proof}

\begin{theorem}\label{thm:multipartite}
Let $p\ge2$ and $1\le n_1\le\cdots\le n_p$. Then
\begin{equation}\label{eq:multipartite}
\psq(K_{n_1,\ldots,n_p})=\max\{2,n_p\}.
\end{equation}
In particular, the value is $n_p$ whenever $n_p\ge2$.
\end{theorem}
\begin{proof}
If $n_p=1$, the graph is $K_p$ and Proposition~\ref{prop:complete} applies. Assume $n_p\ge2$. The largest part gives $\psq\ge n_p$. For the upper bound, take a nonempty induced complete multipartite graph $H$, let $N=|V(H)|$, and let $r$ be its largest part size. If $r=1$, then $H$ is complete and $\pq(H)\le2$. Suppose $r\ge2$. Every vertex has at least $N-r$ neighbors, so every class of a quorum coloring has size at least $(N-r+1)/2$. If $N>r+1$, its number of classes is at most
\[
\frac{2N}{N-r+1}<r+1,
\]
where the strict inequality is equivalent to $(r-1)N>(r-1)(r+1)$. Hence the number is at most $r$. If $N=r$, the graph is edgeless and the bound is immediate. If $N=r+1$, it is a star with $r$ leaves, and $r+1$ singleton colors violate the condition at its center. Thus in all cases $\pq(H)\le\max\{2,r\}\le n_p$. Equation~\eqref{eq:induced} proves the result.
\end{proof}

\subsection{Corona and join graphs}
For graphs $G$ and $H$, the \emph{corona} $G\circ H$ is obtained from one copy of $G$ and $|V(G)|$ disjoint copies of $H$, joining the $i$th vertex of $G$ to every vertex in the $i$th copy of $H$.  The \emph{join} $G+H$ is obtained from the disjoint union of $G$ and $H$ by adding every edge between $V(G)$ and $V(H)$.

\begin{theorem}[Corona graphs]\label{thm:corona}
For all positive integers $n$ and $m$,
\[
\psq(K_n\circ\overline{K_m})
=mn+\left\lceil\frac{m}{n}\left\lfloor\frac{n}{m}\right\rfloor\right\rceil
=\begin{cases}
mn, & 1\le n<m,\\
mn+1, & 1\le m\le n.
\end{cases}
\]
\end{theorem}

\begin{proof}
Set $G=K_n\circ\overline{K_m}$.  Let $L$ be the set of the $mn$ vertices belonging to the copies of $\overline{K_m}$ and let $T=V(G)\setminus L$ be the central clique.  Coloring every vertex of $L$ with a distinct color gives $mn$ colors.  If $n\ge m$, color in addition all vertices of $T$ with one common color.  A vertex of $L$ then sees only itself and its central neighbor, while a vertex of $T$ has $n$ vertices of its own color in a closed colored neighborhood of size $n+m$.  Hence
\[
\psq(G)\ge
\begin{cases}
mn,&n<m,\\
mn+1,&n\ge m.
\end{cases}
\]

For the upper bound, take an optimal coloring, whose color classes are connected. Let $U$ be its colored central vertices and put $u=|U|$. Every color class avoiding $U$ is a singleton leaf. If $u=0$, there are at most $mn$ colors. Otherwise let $U_1,\ldots,U_t$ be the nonempty intersections of the central color classes with $U$, with $s_j=|U_j|$.

For every central vertex $i$, colored or not, let $p_i$ be the number of its leaves that form singleton color classes. Thus $0\le p_i\le m$ and the total number of colors is
\[
k=t+\sum_{i=1}^n p_i.
\]
For a central class $j$, put $d_j=\sum_{i\in U_j}(m-p_i)$. If $d_j=0$, each center in $U_j$ has all its $m$ leaves colored as singletons. Its closed colored neighborhood has size $u+m$, and exactly $s_j$ of these vertices have its color. The quorum condition therefore gives $2s_j\ge u+m$, in particular $s_j>u/2$. At most one central class can satisfy this. Since the $d_j$ are nonnegative integers, $\sum_jd_j\ge t-1$, and
\[
k\le mn+t-\sum_jd_j\le mn+1.
\]
If $n<m$, then $u\le n<m$ and $2s_j\le2u<u+m$ for every $j$. Hence every $d_j$ is at least one, and the same count gives $k\le mn$. This proves both cases, including $n=m=1$.
\end{proof}

\begin{corollary}\label{cor:corona-quorum}
For all integers $n\ge m\ge1$,
\[
\pq(K_n\circ\overline{K_m})=mn+1.
\]
\end{corollary}
\begin{proof}
When $n\ge m$, the $mn+1$-color construction in the proof of Theorem~\ref{thm:corona} colors every vertex, and is therefore a quorum coloring.  Hence $\pq\ge mn+1$.  Since every quorum coloring is a sub-quorum coloring, Theorem~\ref{thm:corona} gives $\pq\le\psq=mn+1$.
\end{proof}

The next lemma gives the quorum bound needed for the join family.
\begin{lemma}\label{lem:join-quorum}
For all integers $m\ge2$ and $n\ge2$,
\[
\pq(K_m+\overline{K_n})\le2.
\]
\end{lemma}
\begin{proof}
Let $\pi$ be a quorum coloring of $K_m+\overline{K_n}$.  Every color class containing a vertex of $K_m$ has size at least
$\lceil(m+n)/2\rceil$, because each clique vertex is adjacent to every other vertex of the graph.  Hence at most two color classes can meet $K_m$.

Moreover, when $m\ge2$, no color class can lie entirely in $\overline{K_n}$: a vertex in such a class has no same-colored neighbor in the independent part, whereas its closed neighborhood contains itself and all $m$ clique vertices, so the quorum inequality would require $1\ge(m+1)/2$.  Thus every color class meets $K_m$, and $|\pi|\le2$.
\end{proof}

\begin{theorem}[Join graphs]\label{thm:join}
For all integers $m\ge1$ and $n\ge2$,
\[
\psq(K_m+\overline{K_n})=n.
\]
\end{theorem}
\begin{proof}
The independent set $V(\overline{K_n})$ gives $\psq(K_m+\overline{K_n})\ge n$ by assigning a distinct color to each of its vertices and leaving the clique uncolored.

For the reverse inequality, use the induced-subgraph characterization~\eqref{eq:induced}.  Every nonempty induced subgraph of $K_m+\overline{K_n}$ is an independent graph, a complete graph, a star with at least two leaves, or a graph $K_r+\overline{K_s}$ with $r,s\ge2$.  Independent induced subgraphs have at most $n$ vertices; complete induced subgraphs have quorum coloring number at most $2$; stars with at least two leaves have quorum coloring number at most the number of their leaves by Proposition~\ref{prop:star}, hence at most $n$; and when $r,s\ge2$, Lemma~\ref{lem:join-quorum} gives quorum coloring number at most $2$.  The case $s=1$, including $K_2$, is covered by the complete graphs, and the case $r=1,s\ge2$ by the stars.  Thus every induced subgraph has quorum coloring number at most $n$, and~\eqref{eq:induced} yields $\psq(K_m+\overline{K_n})\le n$.
\end{proof}

\section{Rectangular grid strips}\label{sec:grids}
The NP-completeness result above already holds for bipartite graphs. Rectangular grids form a much more structured bipartite family, and for the first fixed widths the sub-quorum coloring number can nevertheless be determined exactly. Write $G_{m,n}=P_m\square P_n$. Dissociation sets are complements of vertex sets meeting every three-vertex path, so
\[
\beta_2(G)=|V(G)|-\tau_3(G),
\]
where $\tau_3(G)$ is the minimum $3$-path vertex-cover number. The exact grid values proved in \cite{Bresar2013,JakovacTaranenko2013} therefore give the following formula.

For later use set
\begin{equation}\label{eq:Fmn}
F(m,n)=
\max\left\{
\left\lceil\frac m2\right\rceil\left\lceil\frac{2n}3\right\rceil+
\left\lfloor\frac m2\right\rfloor\left\lfloor\frac n3\right\rfloor,
\left\lceil\frac n2\right\rceil\left\lceil\frac{2m}3\right\rceil+
\left\lfloor\frac n2\right\rfloor\left\lfloor\frac m3\right\rfloor
\right\}.
\end{equation}

\begin{theorem}[Equivalent form of the grid $3$-path-cover formula \cite{Bresar2013,JakovacTaranenko2013}]\label{thm:gridbeta}
For all integers $m,n\ge1$,
\[
\boxed{\beta_2(G_{m,n})=F(m,n).}
\]
Consequently $\psq(G_{m,n})\ge F(m,n)$.
\end{theorem}
\begin{proof}
The identity is the complementary form of the exact $3$-path vertex-cover formula in \cite[Theorem~2.1, p.~95]{JakovacTaranenko2013}; the relation between dissociation and $3$-path vertex covers is also used in \cite{Bresar2013}. The relation $\psq(G)\ge\beta_2(G)$ then gives the second assertion. For completeness, the lower construction can be seen directly: select
\[
\{v_{i,j}: i\text{ odd and }j\not\equiv0\pmod3\}\cup
\{v_{i,j}: i\text{ even and }j\equiv0\pmod3\},
\]
and also consider its transpose. Each selected vertex has at most one selected neighbor, and the two orientations have the two cardinalities in \eqref{eq:Fmn}.
\end{proof}

The preceding identity suggests that the dissociation lower bound may in fact be exact on every rectangular grid. We record this as the main open problem for this family.

\begin{conjecture}[Rectangular-grid conjecture]\label{conj:grid}
For all integers $m,n\ge2$,
\[
\boxed{\psq(G_{m,n})=\beta_2(G_{m,n})=F(m,n).}
\]
\end{conjecture}

The results below verify Conjecture~\ref{conj:grid} for the first infinite families of grid strips, namely whenever $2\le m\le11$ (and, by symmetry, whenever $2\le n\le11$).

The ladder case has a short direct proof.
\begin{theorem}[Ladder grids]\label{thm:ladder}
For every integer $n\ge1$,
\[
\boxed{\psq(G_{2,n})=\beta_2(G_{2,n})=2\left\lceil\frac n2\right\rceil.}
\]
\end{theorem}
\begin{proof}
Selecting both vertices in every odd column gives a dissociation set of size $2\lceil n/2\rceil$, hence the same lower bound for $\psq$.

For the reverse inequality, use Lemma~\ref{lem:omega} and a feasible representative pair $(T,M)$. Put $P=V(M)$ and $B=V\setminus(T\cup P)$. Let $a,b,u,v$ denote, respectively, the numbers of columns of types $TT,BB,TP,PB$, allowing either vertical order in the last two types. Then
\[
2(|T|+|M|)-2n=|T|-|B|=2(a-b)+u-v.
\]
In a $TP$ column the $P$ vertex must be matched horizontally: it cannot be matched vertically because the other vertex lies in $T$. Its partner lies in a $PB$ column, since the $T$ vertex in the first column already has its vertical occupied neighbor and therefore its horizontal neighbor cannot be occupied. Distinct $TP$ columns have distinct partners, so $u\le v$.

Moreover, every column adjacent to a $TT$ column is $BB$, because each of the two $T$ vertices already has its vertical occupied neighbor. On a path of columns this implies $a\le b$, except for the alternating pattern $TT,BB,\ldots,BB,TT$ when $n$ is odd, in which case $a=b+1$ and $u=v=0$. Hence $|T|+|M|\le n$ for even $n$ and $|T|+|M|\le n+1$ for odd $n$. Lemma~\ref{lem:omega} proves the upper bound, and the dissociation construction attains it.
\end{proof}

For widths $3$ through $11$ we use an exact finite-state transfer computation. Together with the ladder case, the next theorem verifies Conjecture~\ref{conj:grid} for all strip widths from $2$ through $11$.

\begin{theorem}[Computer-assisted exact grid strips]\label{thm:gridstrips}
For every $n\ge1$,
\[
\begin{aligned}
\psq(G_{3,n})=\beta_2(G_{3,n})&=\left\lceil\frac{5n}{3}\right\rceil,\\
\psq(G_{4,n})=\beta_2(G_{4,n})&=\begin{cases}2n,&n\text{ even},\\2n+1,&n\text{ odd},\end{cases}\\
\psq(G_{5,n})=\beta_2(G_{5,n})&=\max\left\{3\left\lceil\frac{2n}{3}\right\rceil+2\left\lfloor\frac n3\right\rfloor,
4\left\lceil\frac n2\right\rceil+\left\lfloor\frac n2\right\rfloor\right\},\\
\psq(G_{6,n})=\beta_2(G_{6,n})&=\begin{cases}3n,&n\text{ even},\\3n+1,&n\text{ odd},\end{cases}\\
\psq(G_{7,n})=\beta_2(G_{7,n})&=\max\left\{4\left\lceil\frac{2n}{3}\right\rceil+3\left\lfloor\frac n3\right\rfloor,
5\left\lceil\frac n2\right\rceil+2\left\lfloor\frac n2\right\rfloor\right\}.
\end{aligned}
\]
For each $m\in\{8,9,10,11\}$,
\[
\boxed{\psq(G_{m,n})=\beta_2(G_{m,n})=F(m,n).}
\]
Together with Theorem~\ref{thm:ladder}, this determines $\psq(G_{m,n})$ for every fixed width $2\le m\le11$.
\end{theorem}
\begin{proof}
Theorem~\ref{thm:gridbeta} gives the displayed values as lower bounds. For the upper bound it is enough, by Lemma~\ref{lem:omega}, to maximize $\Omega(G_{m,n})$.

For fixed $m$, process the grid one column at a time. A frontier state records whether each exposed vertex is unused, belongs to $T$ with zero or one already exposed occupied neighbor, is a completed matching endpoint, or is a matching endpoint awaiting its partner in one of the unprocessed directions. A new occupied vertex consumes one unit of the available neighbor allowance of every adjacent frontier vertex of $T$; an awaiting matching endpoint forces its partner. Assign weight two to a vertex of $T$ and weight one to each matching endpoint. These local transitions enumerate exactly the feasible partial representative pairs, and the maximum terminal doubled weight is $2\Omega(G_{m,n})$.

Let $V_n^{(m)}$ be the vector of best doubled weights after $n$ columns, with $-\infty$ for unreachable states. The one-column transition $A_m$ is max-plus linear, so $A_m(V+c)=A_m(V)+c$. Exact integer computation gives the following complete-vector translation certificates:
\begin{center}
\begin{tabular}{rrrrr}
\toprule
$m$ & $n_0$ & period $p$ & increment $\Delta$ & reachable states\\
\midrule
3 & 3 & 3 & 10 & 65\\
4 & 3 & 2 & 8 & 236\\
5 & 9 & 3 & 16 & 856\\
6 & 6 & 2 & 12 & 3105\\
7 & 10 & 3 & 22 & 11263\\
8 & 9 & 2 & 16 & 40855\\
9 & 15 & 3 & 28 & 148196\\
10 & 12 & 2 & 20 & 537561\\
11 & 16 & 3 & 34 & 1949930\\
\bottomrule
\end{tabular}
\end{center}
For each row the computed indexed vectors satisfy
\[
V_{n_0+p}^{(m)}=V_{n_0}^{(m)}+\Delta
\]
with identical reachable support. Max-plus linearity propagates this identity to every $n\ge n_0$, giving
\[
\Omega(G_{m,n+p})=\Omega(G_{m,n})+\Delta/2.
\]
The finite initial readouts give $\Omega(G_{m,n})=F(m,n)$ for $1\le n<n_0+p$. Direct substitution in \eqref{eq:Fmn} verifies the same period and increment for $F(m,n)$ from the listed starting point onward. Induction therefore yields $\Omega(G_{m,n})=F(m,n)$ for all $n\ge1$, and Lemma~\ref{lem:omega} completes the proof.
\end{proof}

\begin{remark}\label{rem:grid-computation}
The strip theorem is computer-assisted: the finite part of the proof consists of exact integer frontier-vector computations and equality of the complete indexed vectors at the two certificate endpoints, not merely agreement of scalar optima. The code, indexed endpoint vectors and finite readouts are available in the \href{https://github.com/the-omega-institute/subquorum-colorings/tree/d8f15435f8ea2b4b97ac94008f1611204cbd1c0b}{public reproducibility package at revision d8f1543}. The programs \texttt{develop/grid\_omega.cpp} and \texttt{develop/check\_omega.cpp} independently implement the transitions. The former computes the indexed vectors; the latter reconstructs both endpoints from the empty frontier. Unreachable states have weight $-\infty$ and are omitted from the stored finite support. Terminal acceptance requires that every matching endpoint has its partner; the terminal maximum is over these completed states. Since this terminal rule is fixed, it preserves the displayed vector translation.

From that checkout, \texttt{python3 scripts/verify\_archive.py} checks archived file hashes, the finite readouts and stored endpoint translations. After compiling both C++17 programs into \texttt{build/}, \texttt{python3 scripts/check\_strip.py m} regenerates the readouts and independently reconstructs both endpoints for width $m$. The package's \href{https://github.com/the-omega-institute/subquorum-colorings/blob/d8f15435f8ea2b4b97ac94008f1611204cbd1c0b/docs/REPRODUCE.md}{reproduction guide} gives the complete commands.
\end{remark}

The hypercube coloring identity proved below also has a \href{https://github.com/the-omega-institute/trureturing/blob/a53350d0255768f5c26c8b36bf34d47eeb00296f/D5/S3/Combinatorics/Graph/HypercubeSubQuorum.lean}{Lean~4 formalization at a fixed source revision}, with theorem \texttt{subQuorumChromaticNumber\_hypercube} for $n\ge2$. Its recorded axiom closure is \texttt{propext}, \texttt{Classical.choice} and \texttt{Quot.sound}. This formal statement covers the hypercube coloring number; the grid, tree and subcubic results use the written arguments and integer computations described here.

\section{Hypercubes}\label{sec:cube}
Let $Q_n$ be the $n$-dimensional hypercube with vertex set $\{0,1\}^n$, where two binary strings are adjacent when they differ in exactly one coordinate. The Hamming weight of $x$ is denoted $w(x)$. The graph is $n$-regular and bipartite, with parity classes
\[
E_n=\{x:w(x)\equiv0\pmod2\},\qquad O_n=\{x:w(x)\equiv1\pmod2\},
\]
each of cardinality $2^{n-1}$.

It is useful to contrast the partial and total versions already at this point. Hedetniemi et al.~\cite{HHLM2013} determined the ordinary quorum coloring number of hypercubes. The sub-quorum problem studied here is substantially different because vertices may be left uncolored; the even-parity construction below already gives $2^{n-1}$ sub-quorum colors. Thus the hypercube section is not a reformulation of the known total-coloring result, but an extremal problem created by the additional freedom of partial coloring.

\begin{theorem}\label{thm:cubebeta}
For every integer $n\ge2$,
\[
\boxed{\bii(Q_n)=2^{n-1}.}
\]
Consequently, $\psq(Q_n)\ge 2^{n-1}$.
\end{theorem}
\begin{proof}
Either parity class of $Q_n$ is independent and has cardinality $2^{n-1}$, so $\bii(Q_n)\ge2^{n-1}$.

For the reverse inequality, suppose that $S\subseteq V(Q_n)$ is $2$-independent and $|S|>2^{n-1}$. Choose $S'\subseteq S$ with $|S'|=2^{n-1}+1$. Huang's induced-subgraph theorem for the hypercube \cite{Huang2019} states that every subset of more than half the vertices of $Q_n$ induces a subgraph of maximum degree at least $\sqrt n$.  Applied to $S'$, it gives
\[
\Delta(Q_n[S'])\ge \sqrt n.
\]
Since $n\ge2$ and vertex degrees are integers, this implies $\Delta(Q_n[S'])\ge2$. But $S'\subseteq S$ and $\Delta(Q_n[S])\le1$, a contradiction. Hence $|S|\le2^{n-1}$, proving the equality. The final assertion follows from Theorem~\ref{thm:beta2}.
\end{proof}

\subsection{The sub-quorum coloring number of the hypercube}

We now prove that the lower bound in Theorem~\ref{thm:cubebeta} is exact in every dimension.  The proof uses the symmetric signed adjacency matrix introduced by Huang in his proof of the Sensitivity Conjecture \cite{Huang2019}.  We recall only the two properties that are needed here.  There exists a symmetric matrix $H_n$, indexed by $V(Q_n)$, such that
\[
 (H_n)_{xy}\in\{0,\pm1\},\qquad
 |(H_n)_{xy}|=1\ \Longleftrightarrow\ xy\in E(Q_n),
\]
and
\begin{equation}\label{eq:huang-square}
 H_n^2=nI.
\end{equation}

\begin{theorem}\label{thm:cubeall}
For every integer $n\ge2$,
\[
\boxed{\psq(Q_n)=\bii(Q_n)=2^{n-1}.}
\]
\end{theorem}

\begin{proof}
The equality $\bii(Q_n)=2^{n-1}$ was proved in Theorem~\ref{thm:cubebeta}, and Theorem~\ref{thm:beta2} therefore gives
\[
 \psq(Q_n)\ge2^{n-1}.
\]
It remains to prove the reverse inequality.

Let $f$ be an arbitrary sub-quorum coloring of $Q_n$, with colored set $S$, and suppose that it uses $k$ colors.  Divide its color classes into two types.  From every color class containing an edge choose the two endpoints of one such edge, and from every independent color class choose one representative.  Let $A$ be the set of all selected vertices and let $T\subseteq A$ be the set of representatives chosen from the independent color classes.  If there are $r$ classes containing an edge and $s$ independent classes, then
\[
 k=r+s,\qquad |A|=2r+s,\qquad |T|=s,
\]
and consequently
\begin{equation}\label{eq:AT-2k}
 |A|+|T|=2k.
\end{equation}
Put
\[
 B=V(Q_n)\setminus A.
\]

We first record two consequences of the selection.  If $t\in T$, then its color class is independent.  Hence the only vertex of its own color in $N[t]\cap S$ is $t$ itself.  The sub-quorum condition therefore gives
\[
 |N[t]\cap S|\le2.
\]
Since $A\subseteq S$, it follows that
\begin{equation}\label{eq:t-degree-A}
 d_A(t)\le1\qquad(t\in T).
\end{equation}
On the other hand, for every $a\in A$,
\begin{equation}\label{eq:a-degree-T}
 d_T(a)\le n-1.
\end{equation}
Indeed, if $a\in T$, this follows from \eqref{eq:t-degree-A} and $n\ge2$.  If $a\notin T$, then $a$ was selected from a color class containing an edge, together with the other endpoint of the chosen edge.  That selected partner belongs to $A\setminus T$, so among the $n$ neighbors of $a$ at most $n-1$ can belong to $T$.

Consider now an arbitrary vector $x\in\mathbb R^{V(Q_n)}$ supported on $T$.  For $a\in A$, the support property of $H_n$ and Cauchy--Schwarz give
\[
 |(H_nx)_a|^2
 =\left|\sum_{t\in N(a)\cap T}(H_n)_{at}x_t\right|^2
 \le d_T(a)\sum_{t\in N(a)\cap T}x_t^2.
\]
Using \eqref{eq:a-degree-T} and summing over $a\in A$ yields
\[
 \|(H_nx)|_A\|_2^2
 \le (n-1)\sum_{a\in A}\sum_{t\in N(a)\cap T}x_t^2.
\]
Interchanging the order of summation and using \eqref{eq:t-degree-A}, we obtain
\begin{equation}\label{eq:energy-A}
 \|(H_nx)|_A\|_2^2
 \le(n-1)\sum_{t\in T}d_A(t)x_t^2
 \le(n-1)\|x\|_2^2.
\end{equation}

Since $H_n$ is symmetric and satisfies \eqref{eq:huang-square},
\[
 \|H_nx\|_2^2
 =x^{\mathsf T}H_n^2x
 =n\|x\|_2^2.
\]
Together with \eqref{eq:energy-A}, this gives
\begin{equation}\label{eq:energy-B}
 \|(H_nx)|_B\|_2^2
 =\|H_nx\|_2^2-\|(H_nx)|_A\|_2^2
 \ge\|x\|_2^2.
\end{equation}
Thus the linear map
\[
 L:\mathbb R^T\longrightarrow\mathbb R^B,
 \qquad L(x)=(H_nx)|_B,
\]
is injective.  Hence
\[
 |T|\le|B|.
\]
Finally, since $|A|+|B|=2^n$, equation \eqref{eq:AT-2k} gives
\[
 2k=|A|+|T|
 \le |A|+|B|
 =2^n.
\]
Therefore $k\le2^{n-1}$.  Since $f$ was arbitrary,
\[
 \psq(Q_n)\le2^{n-1},
\]
and the lower bound proves the theorem.
\end{proof}

\begin{remark}
The proof controls arbitrary partial colorings, not only the parity construction.  Its key point is the restricted energy estimate \eqref{eq:energy-A}: the sub-quorum condition forces every representative in $T$ to have at most one selected neighbor, while the endpoint selected from an edged color class has a selected partner outside $T$.  Huang's identity $H_n^2=nI$ then converts this local information into the global injection $\mathbb R^T\hookrightarrow\mathbb R^B$.
\end{remark}

The lower-bound construction appearing in Theorem~\ref{thm:cubebeta} is illustrated in Figures~\ref{fig:q3parity} and~\ref{fig:q4parity}.  Only the even-weight vertices are colored, and each of them receives its own color; the odd-weight vertices are left uncolored.  Since every hypercube edge joins vertices of opposite parity, the colored support is independent.  Thus these figures display optimal sub-quorum colorings with $2^{n-1}$ colors, rather than a two-coloring of the whole hypercube.

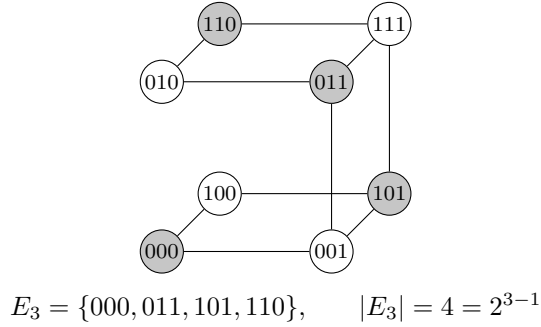
\begin{figure}[ht]
\centering
\begin{tikzpicture}[
  scale=1.02,
  every node/.style={circle,draw,minimum size=5.7mm,inner sep=0pt,font=\scriptsize},
  sel/.style={fill=gray!45},
  un/.style={fill=white}
]
\node[sel] (q000) at (0,0) {000};
\node[un]  (q001) at (2.2,0) {001};
\node[un]  (q010) at (0,2.2) {010};
\node[sel] (q011) at (2.2,2.2) {011};

\node[un]  (q100) at (.75,.75) {100};
\node[sel] (q101) at (2.95,.75) {101};
\node[sel] (q110) at (.75,2.95) {110};
\node[un]  (q111) at (2.95,2.95) {111};

\draw (q000)--(q001)--(q011)--(q010)--cycle;
\draw (q100)--(q101)--(q111)--(q110)--cycle;
\draw (q000)--(q100) (q001)--(q101) (q010)--(q110) (q011)--(q111);

\node[draw=none,rectangle,font=\small] at (1.48,-.72)
{$E_3=\{000,011,101,110\}$,\qquad $|E_3|=4=2^{3-1}$};
\end{tikzpicture}
\caption{An optimal sub-quorum coloring of $Q_3$. Filled vertices are the even-weight vertices and each receives a distinct color; unfilled vertices are uncolored. Hence the colored support is independent and the coloring uses four colors.}
\label{fig:q3parity}
\end{figure}

\begin{figure}[ht]
\centering
\begin{tikzpicture}[
  scale=.88,
  every node/.style={circle,draw,minimum size=4.9mm,inner sep=0pt,font=\tiny},
  sel/.style={fill=gray!45},
  un/.style={fill=white}
]
\node[sel] (a000) at (0,0) {0000};
\node[un]  (a001) at (1.8,0) {0001};
\node[un]  (a010) at (0,1.8) {0010};
\node[sel] (a011) at (1.8,1.8) {0011};
\node[un]  (a100) at (.58,.58) {0100};
\node[sel] (a101) at (2.38,.58) {0101};
\node[sel] (a110) at (.58,2.38) {0110};
\node[un]  (a111) at (2.38,2.38) {0111};
\draw (a000)--(a001)--(a011)--(a010)--cycle;
\draw (a100)--(a101)--(a111)--(a110)--cycle;
\draw (a000)--(a100) (a001)--(a101) (a010)--(a110) (a011)--(a111);

\begin{scope}[xshift=5.0cm]
\node[un]  (b000) at (0,0) {1000};
\node[sel] (b001) at (1.8,0) {1001};
\node[sel] (b010) at (0,1.8) {1010};
\node[un]  (b011) at (1.8,1.8) {1011};
\node[sel] (b100) at (.58,.58) {1100};
\node[un]  (b101) at (2.38,.58) {1101};
\node[un]  (b110) at (.58,2.38) {1110};
\node[sel] (b111) at (2.38,2.38) {1111};
\draw (b000)--(b001)--(b011)--(b010)--cycle;
\draw (b100)--(b101)--(b111)--(b110)--cycle;
\draw (b000)--(b100) (b001)--(b101) (b010)--(b110) (b011)--(b111);
\end{scope}

\draw[densely dashed] (a000)--(b000);
\draw[densely dashed] (a001)--(b001);
\draw[densely dashed] (a010)--(b010);
\draw[densely dashed] (a011)--(b011);
\draw[densely dashed] (a100)--(b100);
\draw[densely dashed] (a101)--(b101);
\draw[densely dashed] (a110)--(b110);
\draw[densely dashed] (a111)--(b111);

\node[draw=none,rectangle,font=\small] at (3.7,-.72)
{$|E_4|=8=2^{4-1}$ distinct colors};
\end{tikzpicture}
\caption{The same optimal construction on $Q_4$, drawn as two copies of $Q_3$. Dashed edges join corresponding vertices in the two copies. Exactly the eight even-weight vertices are filled, each with its own color; all odd-weight vertices are uncolored.}
\label{fig:q4parity}
\end{figure}
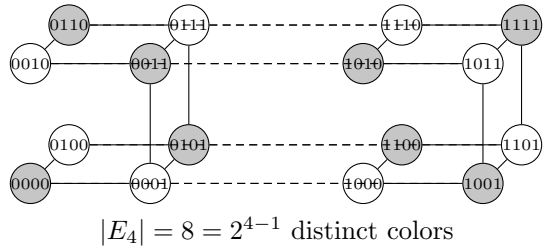

\section{A Nordhaus--Gaddum lower bound}\label{sec:ng}
Following the classical Nordhaus--Gaddum viewpoint \cite{NordhausGaddum1956}, we consider $\psq$ simultaneously on a graph and its complement.  The following logarithmic lower bound is an immediate consequence of the connection with dissociation sets and the classical Ramsey estimate.

\begin{theorem}\label{thm:ng-lower}
For every graph $G$ of order $n\ge2$,
\[
\boxed{\psq(G)+\psq(\overline G)>\log_2 n.}
\]
In particular,
\[
\psq(G)+\psq(\overline G)\ge\lfloor\log_2 n\rfloor+1.
\]
\end{theorem}
\begin{proof}
Let $a=\alpha(G)$ and $b=\omega(G)=\alpha(\overline G)$.  Since every independent set is a dissociation set, Theorem~\ref{thm:beta2} gives
\[
\psq(G)+\psq(\overline G)\ge a+b.
\]
The graph $G$ contains neither an independent set of size $a+1$ nor a clique of size $b+1$, so
\[
n<R(a+1,b+1).
\]
The standard binomial Ramsey bound \cite{Ramsey1930} yields
\[
R(a+1,b+1)\le {a+b\choose a}\le2^{a+b}.
\]
Hence $n<2^{a+b}$ and therefore $a+b>\log_2 n$.
\end{proof}

The complete and empty graphs suggest the complementary extremal question.  Since $\psq(K_n)=2$ for $n\ge2$ and $\psq(\overline{K_n})=n$, they give sum $n+2$.
\begin{conjecture}\label{conj:ng-upper}
For every graph $G$ of order $n\ge2$,
\[
\psq(G)+\psq(\overline G)\le n+2.
\]
\end{conjecture}
Thus the lower and upper Nordhaus--Gaddum extremal functions for $\psq$ remain open beyond Theorem~\ref{thm:ng-lower}.  In particular, determining the true order of
\[
\min_{|V(G)|=n}\{\psq(G)+\psq(\overline G)\}
\]
is a natural Ramsey-type problem.

\section{Declaration on the use of AI-assisted tools}\label{sec:ai}
H.M.\ and W.Z.\ used OpenAI models, primarily GPT-5.6 and GPT-6 Astra, through Codex to assist with exploration of proof strategies and counterexamples, exact search and verifier development, writing and debugging code and Lean proofs, comparison with the literature, and drafting and revising the manuscript. The formalized hypercube coloring theorem was checked by the Lean kernel; the other results are supported by the written proofs and integer computations specified in the paper and its reproducibility package. The authors retain responsibility for the mathematical claims, verification code and final text.

\section{Conclusion}
We have developed the sub-quorum coloring number from several complementary viewpoints. The universal inequality $\psq(G)\ge\beta_2(G)$ links the parameter to dissociation, while the degree-sensitive upper bound and the exact subcubic matching characterization yield the sharp cubic interval $n/2\le\psq(G)\le3n/5$ and a stability identity describing colorings near the upper endpoint.

For trees, the edge representation removes color names and leads to a linear-time dynamic program for the optimum value and an optimal coloring. The example $S_{2,2}$ shows that $\psq$ and $\beta_2$ need not coincide even on trees. The complexity result shows, in contrast, that the unrestricted problem remains NP-complete already on bipartite graphs. Within the structured bipartite families, the ladder formula and the exact transfer certificates verify Conjecture~\ref{conj:grid} for every fixed strip width $2\le m\le11$: in these cases $\psq(G_{m,n})=\beta_2(G_{m,n})=F(m,n)$. We conjecture that the same identity holds for every rectangular grid. The hypercube identity $\psq(Q_n)=\beta_2(Q_n)=2^{n-1}$ provides a further infinite bipartite family in which the dissociation lower bound is attained.

The Nordhaus--Gaddum inequality $\psq(G)+\psq(\overline G)>\log_2 |V(G)|$ shows that complementation forces a nonconstant amount of sub-quorum structure. Determining the corresponding upper extremum, and in particular Conjecture~\ref{conj:ng-upper}, remains open. Other natural directions include characterizing trees for which $\psq(T)=\beta_2(T)$, extending the tree algorithm to bounded-treewidth classes, and developing approximation or parameterized algorithms for the general problem.

\end{document}